\documentclass[12pt]{amsart}

\usepackage{amsmath, amssymb, graphics}

\newcommand{\mathsym}[1]{{}}
\newcommand{\unicode}[1]{{}}
\newtheorem{thm}{Theorem}[section]
\newtheorem{cor}[thm]{Corollary}

\newtheorem{prop}[thm]{Proposition}
\theoremstyle{definition}

\newtheorem{defn}[thm]{Definition}
\newtheorem{rem}[thm]{Remark}

\newtheorem*{defn*}{Definition}
\newtheorem*{rems*}{Remarks}
\newtheorem*{rem*}{Remark}

\numberwithin{equation}{section}
\begin{document}

\title [Singularities of affine equidistants] {Singularities of affine equidistants:\\ projections and contacts.}

\author[W. Domitrz]{W. Domitrz}
\address{Warsaw University of Technology, Faculty of Mathematics and Information Science, ul. Koszykowa 75, 00-662 Warszawa, Poland}
\email{domitrz@mini.pw.edu.pl}
\author[P. de M. Rios]{P. de M. Rios}
\address{Departamento de Matem\'atica, ICMC, Universidade de S\~ao Paulo; S\~ao Carlos, SP, 13560-970, Brazil}
\email{prios@icmc.usp.br}
\author[M. A. S. Ruas]{M. A. S. Ruas}
\address{Departamento de Matem\'atica, ICMC, Universidade de S\~ao Paulo; S\~ao Carlos, SP, 13560-970, Brazil}
\email{maasruas@icmc.usp.br}

\thanks{ W. Domitrz was partially supported by NCN, 
P. de M. Rios was partially supported by FAPESP grant no. 2010/15179-8, M. A. S. Ruas  was partially supported by CNPq grant no. 305651/2011-0 and FAPESP grant no. 2008/54222-6.}

\maketitle

\begin{abstract} Using standard methods for studying singularities of projections and of contacts, we classify the stable singularities of affine $\lambda$-equidistants of $n$-dimensional closed submanifolds of $\mathbb R^q$, for $q\leq 2n$, whenever $(2n,q)$ is a pair of nice dimensions \cite{Mather}.
\end{abstract}

\section{Introduction} When $M$ is a smooth closed curve on the affine plane $\mathbb R^2$, the set of all midpoints of chords connecting pairs of points on $M$ with parallel tangent vectors is called the {\it Wigner caustic} of $M$, or the {\it area evolute} of $M$, or still, the {\it affine $1/2$-equidistant} of $M$, denoted $E_{1/2}(M)$.

The $1/2$-equidistant is generalized to any $\lambda$-equidistant, denoted $E_{\lambda}(M)$, $\lambda\in\mathbb R$, by considering all chords  connecting pairs of points of $M$ with parallel tangent vectors and the set of all points of these chords  which stand in the $\lambda$-proportion to their corresponding pair of points on $M$.  In this case, when $M$ is a curve on $\mathbb R^2$, the local classification of stable singularities of $E_{\lambda}(M)$ is well known \cite{Ber, GZ1}.

The definition of the affine $\lambda$-equidistant of $M$ is generalized to the cases when $M$ is an $n$-dimensional closed submanifold of $\mathbb R^q$, with
$q\leq 2n$, by considering the set of all $\lambda$-points of chords connecting pairs of points on $M$ whose direct sum of tangent spaces do not coincide with $\mathbb R^q$, the so-called {\it weakly parallel pairs} on $M$.

In addition to curves in $\mathbb R^2$, the possible stable singularities of $E_{\lambda}(M)$ have been previously studied in the general setting when $M$ is a hypersurface \cite{GZ1, GZ2}, or when $M$ is a surface in $\mathbb R^4$ \cite{GJ}. The cases of curves in $\mathbb R^2$ and surfaces in $\mathbb R^4$ have also been  studied in the particular setting of Lagrangian submanifolds of affine symplectic spaces \cite{DRs}.

In this paper, we classify the possible  stable singularities of $E_{\lambda}(M)$ in a quite more general circumstance, namely, when the double dimension of $M$, $2n$, and the dimension of the ambient affine space, $q$, form a pair of {\it nice dimensions} \cite{Mather}, see Theorem \ref{nicedim} below.

In order to obtain such a classification, we start in Section 2 by defining an affine $\lambda$-equidistant of $M^n\subset \mathbb R^q$ as the set of critical values of  the $\lambda$-point map (projection) $\pi_{\lambda}:\mathbb R^q\times\mathbb R^q\to\mathbb R^q, (x^+,x^-)\mapsto\lambda x^++(1-\lambda)x^-$ restricted to $M\times M$, thus locally a map $$\tilde{\pi}_{\lambda}:\mathbb R^{2n}\to\mathbb R^q \ ,$$ see Definition
\ref{pilambdamap}, Remark \ref{pilambda} and equation (\ref{tildepi}), below. Then, we also present the characterization of affine equidistants by a contact map,
extending previous construction for the Wigner caustic (\cite{OH, GJ}).

In Section 3 we review the standard $\mathcal K$-equivalence and the classification of  $\mathcal K$-simple singularities \cite{Go2, Mather}, Theorem \ref{k-simple} below. Then,  in Section 4 we
 combine the study of singularities of projections and of contacts, in view of Theorem \ref{main} below (\cite{Mather, Mart}), with emphasis on contact reduction to
  rank $0$ map-germs, Proposition 4.14.

 Our main result is obtained in Section 5. First, in Theorem \ref{genAstable} we apply the Multijet Transversality Theorem \cite{GG} to a $\mathcal K$-invariant stratification of the jet space.  When $(2n,q)$ is a pair of nice dimensions, the relevant strata of this stratification are  the $\mathcal K$-simple orbits in jet space. Then, we use the results of Section 4 in the context of affine equidistants: Proposition  \ref{localrings} and Corollary \ref{locrings2}, as well as equations (\ref{1})-(\ref{codimmax}).
 The following table summarizes our main result, Theorem \ref{mainresult}, which is presented more extensively as subsection \ref{result}. The normal forms for  the $\mathcal A$-stable singularities of the map  $\tilde{\pi}_{\lambda}$ follow the notation of \cite{Go2} (see Theorem \ref{k-simple} below) for the $\mathcal K$-simple  rank-$0$ contact map-germ $$\theta_{\lambda}: (\mathbb R^k,0)\to (\mathbb R^{k-(2n-q)},0) \ ,$$ where $k$ is the degree of parallelism of the pair of points on $M$ joined by the chord (cf. Definition  \ref{parallelism}  and Tables I, II, III in Theorem \ref{k-simple}).
\begin{center}
\begin{table}[h]
    \begin{small}
    \noindent
    \begin{tabular}{|c|c|c|clcl}
            \hline
(n ,  q) & Stable $E_{\lambda}(M)$, $M^n\subset\mathbb R^q$ & Restrictions \\ \hline
(1 , 2) & $A_{\mu}$ & $\mu\leq 2$ \\ \hline
(2 , 3) & $A_{\mu}$ & $\mu\leq 3$ \\ \hline
(2 , 4) &  $A_{\mu}, C^{\pm}_{2,2}$ & $\mu\leq 4$ \\ \hline
(3 , 4) & $A_{\mu}, D^{\pm}_{4}$ & $\mu\leq 4$ \\ \hline
(3 , 5) & $A_{\mu}, D^{\pm}_{4}, D^{\pm}_{5}, S_5$  & $\mu\leq 5$ \\ \hline
(3 , 6) & $A_{\mu}, C^{\pm}_{\rho,\tau},  C_6$ & $\mu\leq 6$, $2\leq\rho\leq\tau$,  $\rho+\tau\leq 6$ \\ \hline
(4 , 5) & $A_{\mu}, D^{\pm}_{4}, D^{\pm}_{5}$ & $\mu\leq 5$  \\ \hline
(4 , 7) & $A_{\mu}, D^{\pm}_{\nu}, E_6, E_7, S_{\beta}, T_7, \tilde{T}_7$ & $\mu\leq 7$, $4\leq\nu\leq 7$,   $5\leq\beta\leq7$\\ \hline
(4 , 8) & $A_{\mu}, C^{\pm}_{\rho,\tau},  C_6, C_8, F_7, F_8$  & $\mu\leq 8$, $2\leq\rho\leq\tau$, $\rho+\tau\leq 8$ \\ \hline
(5 , 6) & $A_{\mu}, D^{\pm}_{\nu}, E_6$ & $\mu\leq 6$, $4\leq \nu\leq 6$  \\ \hline

\end{tabular}
\end{small}
\end{table}
\end{center}

We note that the case $M^4\subset\mathbb R^6$ is absent from the table of results. This is because $(2n=8,q=6)$ is not a pair of nice dimensions (see Theorem \ref{nicedim} below). Similarly,  $(2n, q>6)$ is not  a pair of nice dimensions, for all $ n\geq 5$.  Classification of stable singularities of $E_{\lambda}(M)$, in these cases, lies outside the scope of this paper.

As mentioned before, the cases in the table of results when $(n,q)\in \{(1,2),(2,3),(3,4),(4,5),(5,6)\}$ correspond to hypersurfaces  and have been previously studied in \cite{GZ1,GZ2}, while the case $(n,q)=(2,4)$ was partially studied in \cite{GJ}. On the other hand, the results for the cases when $(n,q)\in\{(3,5),(3,6),(4,7),(4,8)\}$ are entirely new.

We emphasize that, in all of the above, we are excluding the cases of {\it vanishing chords}, that is, when the $\lambda$-point of the chord connecting two points on $M$ touches $M$ because the pair of points on $M$
lies in the diagonal of $M\times M$. Such ``diagonal singularities'' or {\it singularities on shell} for $E_{\lambda}(M)$ possess additional symmetries when $\lambda=1/2$ and these have been studied for the cases of curves on the plane and surfaces in $\mathbb R^4$, both in the general setting \cite{GJ} and in the more particular setting of Lagrangian submanifolds of affine symplectic space \cite{DRs2}. In this paper, we don't study  such singularities on shell for $E_{\lambda}(M)$.

\section{Affine equidistants}

\subsection{Definition of affine equidistants}

Let $M$ be a smooth closed $n$-dimensional submanifold of
the affine space $\mathbb R^{q}$, with $q\leq 2n$. Let $a, b$ be points of $M$ and denote by
$\tau_{a-b}:\mathbb R^q \ni x\mapsto x+(a-b) \in \mathbb R^q$ the
translation by the vector $(a-b)$.
\begin{defn}\label{parallelism} A pair of points $a, b \in M$ ($a\ne b$) is called a
{\bf weakly parallel} pair if
$$T_aM + \tau_{a-b}(T_bM)\ne \mathbb R^q.$$
$\text{codim}(T_aM + \tau_{a-b}(T_bM))$ in $T_a\mathbb R^q$ is
called the {\bf codimension of a weakly parallel pair $a, b$}. We
denote it by $\text{codim}(a,b)$.

A weakly parallel pair $a, b \in M$ is
called {\bf $k$-parallel} if
\begin{equation}\label{k//} \dim(T_aM \cap \tau_{b-a}(T_bM))=k.\end{equation}
If $k=n$ the pair $a, b \in M$ is called {\bf strongly parallel},
or just {\bf parallel}. We also refer to $k$ as the {\bf degree of
parallelism} of the pair $(a,b)$ and denote it by
$\text{deg}(a,b)$.
The degree of parallelism and the codimension of parallelism are
related in the following way:
\begin{equation}\label{deg-codim}
2n-\text{deg}(a,b)=q-\text{codim}(a,b).
\end{equation}
\end{defn}

\begin{defn}
A {\bf chord} passing through a pair $a,b$, is the line
$$
l(a,b)=\{x\in \mathbb R^q|x=\lambda a + (1-\lambda) b, \lambda \in
\mathbb R\}.
$$
\end{defn}
\begin{defn} For a given $\lambda$, an {\bf affine
$\lambda$-equidistant} of $M$, $E_{\lambda}(M)$, is the set of all
$x\in \mathbb R^q$ such that $x=\lambda a + (1-\lambda) b$, for
all weakly parallel pairs $a,b \in M$. $E_{\lambda}(M)$ is also
called a (affine) {\bf momentary equidistant} of $M$.
Whenever $M$ is understood, we write
$E_{\lambda}$ for $E_{\lambda}(M)$.
\end{defn}
Note that, for any
$\lambda$, $E_{\lambda}(M)=E_{1-\lambda}(M)$ and in particular
$E_0(M)=E_1(M)=M$. Thus, the case $\lambda=1/2$ is special:
\begin{defn}$E_{1/2}(M)$ is called the {\bf Wigner caustic} of $M$  \cite{Ber, OH}.
\end{defn}

\subsection{Characterization of affine equidistants by projection}

Consider the product affine space: $\mathbb R^{q}\times \mathbb R^{q}$ with coordinates $(x_+,x_-)$
and the tangent bundle to $\mathbb R^{q}$:  $T\mathbb
R^{q}=\mathbb R^{q}\times \mathbb R^{q}$ with coordinate system
$(x,\dot{x})$ and standard projection
$\pi: T\mathbb R^{q}\ni (x,\dot{x})\rightarrow x \in \mathbb R^{q}$.

\begin{defn} For $\lambda\in\mathbb{R}$, a {\bf $\lambda$-chord transformation}
$$\Gamma_{\lambda}:\mathbb{R}^{q}\times\mathbb{R}^{q}\to T\mathbb{R}^{q} \ , \ (x^+,x^-)\mapsto(x,\dot{x})$$
is a  linear diffeomorphism  defined by
the {\it $\lambda$-point equation}:
\begin{equation}\label{x}
x=\lambda x^+ + (1-\lambda)x^- \ ,
\end{equation}
for the $\lambda$-point $x$, and a {\it chord equation}:
\begin{equation}\label{x.}
\dot{x}=x^+ - x^-.
\end{equation}
\end{defn}

\begin{rem} For our purposes, the choice  (\ref{x.}) for a chord equation is not unique, but is the simplest one. Among other possibilities,
the choice $\dot{x}=\lambda x^+ - (1-\lambda)x^-$ is particularly well suited for the study of affine equidistants of {\it Lagrangian} submanifolds in symplectic space \cite{DRs}.
\end{rem}

Now, let $M$ be a smooth closed $n$-dimensional submanifold of the
affine space $\mathbb R^{q}$ ($2n\ge q$) and consider the product
$M\times M\subset \mathbb R^{q}\times\mathbb R^{q}$. Let $\mathcal
M_{\lambda}$ denote the image of $M\times M$ by
a  $\lambda$-chord transformation,
$$\mathcal M_{\lambda} = \Gamma_{\lambda}(M\times M) \ ,$$
which is a $2n$-dimensional smooth submanifold of $T\mathbb
R^{q}$.

Then we have the following general characterization:
\begin{thm}[\cite{DRs}]\label{gensing}
The set of critical values of the standard projection $\pi:
T\mathbb R^{q}\to\mathbb R^{q}$ restricted to $\mathcal
M_{\lambda}$ is $E_{\lambda}(M)$.
\end{thm}

\begin{defn}\label{pilambdamap} For $\lambda\in\mathbb{R}$, the {\bf $\lambda$-point map} is the projection
$$\pi_{\lambda}:  \mathbb R^{q}\times\mathbb R^{q}\to\mathbb R^{q} \ , \ (x^+,x^-)\mapsto x=\lambda x^+ + (1-\lambda)x^- \ .$$
\end{defn}

\begin{rem}\label{pilambda}
Because $\pi_{\lambda}=\pi\circ\Gamma_{\lambda}$ we can rephrase Theorem \ref{gensing}: {\it the  set of critical values of the projection $\pi_{\lambda}$ restricted to $M\times M$ is $E_{\lambda}(M)$}.
\end{rem}

\subsection{Characterization of affine equidistants by contact}

In the literature, if  $M\subset\mathbb R^{2}$ is a smooth curve, the Wigner caustic $E_{1/2}(M)$
has been described in various ways. A particular description says that,
if ${\mathcal R}_a:\mathbb R^{2}\to\mathbb R^{2}$
denotes reflection through $a\in\mathbb R^{2}$, then $a\in E_{1/2}(M)$ when $M$ and  ${\mathcal R}_a(M)$ are not
transversal \cite{Ber, OH}. This description has also been used in \cite{OH} for the case of Lagrangian surfaces in symplectic  $\mathbb R^{4}$ and,
more recently \cite{GJ}, for the case of general surfaces in $\mathbb R^{4}$.

We now generalize this description for every $\lambda$-equidistant of submanifolds of more arbitrary dimensions.

\begin{defn}\label{reflection}
For $\lambda\in\mathbb{R}\setminus  \{0,1\}$, a  $\lambda$-{\bf reflection} through $a\in  \mathbb R^{q}$ is the map
\begin{equation}\label{reflect}
{\mathcal R}_a^{\lambda}: \mathbb R^{q}\to\mathbb R^{q} \ , \ x\mapsto {\mathcal R}_a^{\lambda}(x)=\frac{1}{\lambda}a-\frac{1-\lambda}{\lambda}x
\end{equation}
\end{defn}
\begin{rem} A  $\lambda$-reflection through $a$ is not a reflection in the strict sense because ${\mathcal R}_a^{\lambda}\circ{\mathcal R}_a^{\lambda}\neq id: \mathbb R^{q}\to\mathbb R^{q}$, instead,
$${\mathcal R}_a^{1-\lambda}\circ{\mathcal R}_a^{\lambda}= id: \mathbb R^{q}\to\mathbb R^{q} \ ,$$
so that, if $a=a_{\lambda}=\lambda a^+ + (1-\lambda)a^-$ is the $\lambda$-point of $(a^+,a^-)\in \mathbb R^{2q}$,
$${\mathcal R}_{a_{\lambda}}^{\lambda}(a^-)=a^+ \ , \  {\mathcal R}_{a_{\lambda}}^{1-\lambda}(a^+)=a^- \ .$$
Of course, for $\lambda=1/2$,  ${\mathcal R}_a^{1/2} \equiv {\mathcal R}_a$ is a reflection in the strict sense.
\end{rem}

Now, let $M$ be a smooth $n$-dimensional submanifold of $\mathbb R^{q}$, with $2n\geq q$, and let $a=a_{\lambda}=\lambda a^+ + (1-\lambda)a^-$ be the $\lambda$-point of $(a^+,a^-)\in M\times M\subset \mathbb R^{q}\times\mathbb R^{q}$. Also,  let
$M^+$ be a germ of submanifold $M$ around $a^+$ and $M^-$
be a germ of submanifold $M$ around $a^-$.  We have:

\begin{prop}
The following statements are equivalent:

\noindent (i) The $\lambda$-point $a$ belongs to $E_{\lambda}(M)$.

\noindent (ii) $M^+$ and ${\mathcal R}_{a}^{\lambda}(M^-)$ are not transversal at $a^+$.

\noindent (iii) $M^-$ and ${\mathcal R}_{a}^{1-\lambda}(M^+)$ are not transversal at $a^-$.
\end{prop}

\begin{rem} Furthermore, from Remark \ref{pilambda} we see that the study of the singularities of affine equidistants is the study of the singularities of $\pi_{\lambda}$.
But this is the same as the study of the singularities at $a=0$ of
$$(x^+,x^-)\to x^+ +\frac{1-\lambda}{\lambda}x^-=x^+-{\mathcal R}_0^{\lambda}(x^-) \ .$$
In other words, {\it the study of the singularities of $E_{\lambda}(M)\ni 0$ can be proceeded via the study of the contact between $M^+$ and ${\mathcal R}_0^{\lambda}(M^-)$ or,
equivalently, the contact between  $M^-$ and ${\mathcal R}_0^{1-\lambda}(M^+)$}.
\end{rem}

\section{${\mathcal K}$-equivalence}

We recall some basic definitions and results (for details,
see \cite{AGV}).

Henceforth, $\mathcal E_s$ denotes the local ring of smooth function-germs on $\mathbb R^s$, and $\mathfrak m_s$  its  maximal ideal.

\begin{defn}\label{Keq}
Map-germs $f, \widetilde{f}:(\mathbb R^s,y_0)\rightarrow (\mathbb
R^t,0)$ are $\mathcal K$-{\bf equivalent}  if there
exists a diffeomorphism-germ $\phi:(\mathbb R^s,y_0)\rightarrow
(\mathbb R^s,y_0)$ and a map-germ $A:(\mathbb R^s,y_0)\rightarrow
GL(\mathbb R^t)$ \ such that \
$
\tilde f=A\cdot (f\circ \phi).
$
\end{defn}

\begin{thm}[\cite{AGV}] \label{V-ideal}
For the $\mathcal K$-equivalence of two map-germs it is necessary
and sufficient that two ideals generated by the components of
these map-germs may be mapped one to the other by an
isomorphism of $\mathcal E_s$ induced by a  diffeomorphism-germ of the source space $(\mathbb R^s,y_0)$.
\end{thm}

\begin{defn} A map-germ $F:(\mathbb R^s\times \mathbb
R^p,(y_0,z_0))\rightarrow \mathbb R^t$ is a {\bf deformation} of a
map-germ $f:(\mathbb R^s,y_0)\rightarrow \mathbb R^t$ if
$F|_{\mathbb R^s\times \{z_0\}}=f$, where
$p$ is the number of parameters of deformation $F$.
\end{defn}

\begin{defn} A diffeomorphism-germ $\Phi:(\mathbb R^s\times \mathbb
R^p,(y_0,z_0))\rightarrow (\mathbb R^s\times \mathbb
R^p,(y_0,z_0))$ is called {\bf fiber-preserving} if
$\Phi(y,z)=(Y(y,z),Z(z))$ for a smooth map-germ $Y:(\mathbb
R^s\times \mathbb R^p,(y_0,z_0))\rightarrow (\mathbb R^s,y_0)$ and
a diffeomorphism-germ $Z:( \mathbb R^p,z_0)\rightarrow (\mathbb
R^p,z_0)$. It means that $\Phi$ preserves the fibers of the
projection $pr :(\mathbb R^s\times \mathbb
R^p,(y_0,z_0))\rightarrow (\mathbb R^p,z_0)$.
\end{defn}

\begin{defn}
Deformations $F, \widetilde{F}:(\mathbb R^s\times \mathbb
R^p,(y_0,z_0))\rightarrow (\mathbb R^t,0)$ of respective map-germs $f,
\widetilde{f}:(\mathbb R^s,y_0)\rightarrow (\mathbb R^t,0)$
  are {\bf fiber $\mathcal K$-equivalent} if there
is a fiber-preserving diffeomorphism-germ $\Phi:(\mathbb
R^s\times \mathbb R^p,(y_0,z_0))\rightarrow (\mathbb R^s\times
\mathbb R^p,(y_0,z_0))$, i.e. $\Phi(y,z)=(Y(y,z),Z(z))$, and a
map-germ $\mathbb A:(\mathbb R^s\times \mathbb R^p,(y_0,z_0))\rightarrow
GL(\mathbb R^t)$ \ such that \
$
\tilde F=\mathbb A\cdot (F\circ \Phi).
$
\end{defn}

\begin{cor}\label{V-def-ideal}
For the fiber $\mathcal K$-equivalence of two deformations it is
necessary and sufficient that the two ideals of $\mathcal E_{s+p}$ generated by the
components of these deformations may be mapped one to the
other by an isomorphism of $\mathcal E_{s+p}$ induced by a fiber-preserving diffeomorphism-germ
of the source space $(\mathbb R^s\times \mathbb R^p,(y_0,z_0))$.
\end{cor}

\begin{defn}\label{Ksimp}
The germ $f:(\mathbb R^s,0)\rightarrow (\mathbb R^t,0)$ is said to be
{\bf $\mathcal K$-simple} if its $k$-jet, for any $k$, has a
neighborhood in the  jet space $J^k_{0,0}(\mathbb R^s,
\mathbb R^t)$ that intersects only a finite number of $\mathcal
K$-equivalence classes (bounded by a constant independent of $k$).
\end{defn}

\begin{defn}
The $p$-parameter {\bf suspension} of the map-germ $f:(\mathbb
R^s,0)\rightarrow (\mathbb R^t,0)$ is the map germ
$$
F:(\mathbb R^s\times \mathbb R^p,0)\ni (y,z)\mapsto (f(y),z)\in
(\mathbb R^t\times \mathbb R^p,0).
$$
\end{defn}

\begin{thm}[\cite{Go2}]\label{k-simple}
$\mathcal K$-simple map-germs $(\mathbb R^s,0)\rightarrow (\mathbb
R^t,0)$ with $s\ge t$ belong, up to $\mathcal K$-equivalence and
suspension, to one of the following three lists in Tables
\ref{ADE}-\ref{Giusti}:

\begin{center}
\begin{table}[h]
    \begin{small}
    \noindent
    \begin{tabular}{|c|c|c|}
            \hline
    Notation &  Normal form & Restrictions  \\ \hline

     $A_{\mu}$ & $y_1^{\mu+1}+Q_{s-1}$ & $\mu\ge 1$ \\ \hline
     $D_{\mu}$ & $y_1^2y_2\pm y_2^{\mu-1}+Q_{s-2}$ & $\mu\ge 4$ \\ \hline
     $E_6$ & $y_1^3+y_2^4+Q_{s-2}$ & -\\ \hline
     $E_7$ & $y_1^3+y_1y_2^3+Q_{s-2}$ &  -\\ \hline
     $E_8$ & $y_1^3+y_2^5+Q_{s-2}$ &  -\\ \hline
\end{tabular}
\smallskip
\caption{\small $\mathcal K$-simple germs $\mathbb
R^s\rightarrow \mathbb R$.   $Q_{s-i}=\pm y_{i+1}^2\pm \cdots \pm
y_s^2$.}\label{ADE}
\end{small}
\end{table}
\begin{table}[h]
    \begin{small}
    \noindent
    \begin{tabular}{|c|c|c|}
            \hline
    Notation &  Normal form & Restrictions  \\ \hline

     $C^{\pm}_{k,l}$ & $(y_1y_2, y_1^k\pm y_2^l$) & $l\ge k\ge 2$ \\ \hline

     $\widetilde{C}_{2k}$ & $(y_1^2+y_2^2, y_2^k$) & $k\ge 3$ \\ \hline

$F_{2m+1}$ & $(y_1^2+y_2^3, y_2^m$) & $m\ge 3$ \\ \hline

$F_{2m+4}$ &$(y_1^2+y_2^3, y_1y_2^m)$ & $m\ge 2$ \\ \hline

$G^*_{10}$ & $(y_1^2, y_2^4$) & - \\ \hline

$H^{\pm}_{m+5}$ & $(y_1^2\pm y_2^m, y_1y_2^2)$ & $m\ge 4$ \\
\hline
\end{tabular}
\smallskip
\caption{\small $\mathcal K$-simple germs $\mathbb
R^2\rightarrow \mathbb R^2$. }\label{I}
\end{small}
\end{table}
\begin{table}[h]
    \begin{small}
    \noindent
    \begin{tabular}{|c|c|c|}
            \hline
    Notation &  Normal form & Restrictions  \\ \hline
$S_{\mu}$ & $(\pm y_1^2 \pm y_2^2+y_3^{\mu-3}, y_2y_3$) & $\mu\ge 5$ \\
\hline

$T_{7}$ & $(y_1^2+y_2^3+y_3^3, y_2y_3$) & - \\ \hline

$\widetilde{T}_7$ & $(y_1^2+y_2^2, y_2^2+y_3^2)$ &  - \\ \hline

$T_{8}$ & $(y_1^2+y_2^3\pm y_3^4, y_2y_3$) & - \\ \hline

$T_{9}$ & $(y_1^2+y_2^3+y_3^5, y_2y_3$) & - \\ \hline

$U_{7}$ & $(y_1^2+y_2y_3, y_1y_2+y_3^3$) & - \\ \hline

$U_{8}$ & $(y_1^2+y_2y_3+y_3^3, y_1y_2$) & - \\ \hline

$U_{9}$ & $(y_1^2+y_2y_3, y_1y_2+y_3^4$) & - \\ \hline

$W_{8}$ & $(y_1^2+y_2^3, y_2^2+y_1y_3$) & - \\ \hline

$W_{9}$ & $(y_1^2+y_2y_3^2, y_2^2+y_1y_3$) & - \\ \hline

$Z_{9}$ & $(y_1^2+y_3^3, y_2^2+y_3^3$) & - \\ \hline

$Z_{10}$ & $(y_1^2+y_2y_3^2, y_2^2+y_3^3$) & - \\ \hline
\end{tabular}
\smallskip
\caption{\small $\mathcal K$-simple germs $\mathbb
R^3\rightarrow \mathbb R^2$. }\label{Giusti}
\end{small}
\end{table}
\end{center}
\end{thm}

\begin{defn}
A deformation $F:(\mathbb R^s\times \mathbb R^p, (0,0))\rightarrow
(\mathbb R^t,0)$ of a map-germ $f:(\mathbb R^s,0) \rightarrow
(\mathbb R^t,0)$ is {\bf $\mathcal K$-versal} if any other
deformation $\widetilde F:(\mathbb R^s\times \mathbb R^q,(0,0))\rightarrow
(\mathbb R^t,0)$ of $f$ is of the form
$$\widetilde F(y,z)=\mathbb A(y,z)\cdot F(g(y,z),h(z)),$$
 where $\mathbb A:\mathbb R^s\times
\mathbb R^q\rightarrow GL(\mathbb R^t)$, $g:(\mathbb R^s\times
\mathbb R^q,(0,0))\rightarrow (\mathbb R^s,0)$, $h:(\mathbb R^q,0)
\rightarrow (\mathbb R^p,0)$ are map-germs such that $\mathbb A(0,0)$ is
nondegerate matrix and $g(y,0)=y$.
\end{defn}

\begin{thm}[\cite{AGV}]
$\mathcal K$-versal deformations of $\mathcal K$-equivalent germs
with the same number of parameters are fiber $\mathcal
K$-equivalent.
\end{thm}

\section{Singularities of projection and of contact}

\subsection{Singularities of projection}

In view of Theorem \ref{gensing}, let $M$ and $\widetilde{M}$ be smooth closed $n$-dimensional
submanifolds of $\mathbb R^q$, $q\leq 2n$,  and
$$\mathcal
M_{\lambda}=\Gamma_{\lambda}(M\times M) \ , \ \widetilde{\mathcal
M}_{\lambda}=\Gamma_{\lambda}(\widetilde{M}\times \widetilde{M}) \ ,$$
where $\Gamma_{\lambda}$ is the  $\lambda$-chord
transformation.

For local classification of singularities, we introduce the definition:

\begin{defn}\label{lambdachordequivalence} $E_{\lambda}(M)$ and
$E_{\lambda}(\widetilde{M})$ are
{\bf $\lambda$-chord equivalent} if there
exists a fiber-preserving diffeomorphism-germ of $T\mathbb R^q$
that maps the germ of $\mathcal M_{\lambda}$ to
the germ of $\widetilde{\mathcal M}_{\lambda} \ $
i.e. if the following diagram commutes (vertical arrows
indicate diffeomorphism-germs):
$$
\begin{array}{ccccc}
& \Gamma_{\lambda}|_{M\times M} & & \pi & \\
M\times M&\longrightarrow & T\mathbb R^q & \longrightarrow &
\mathbb R^q \\
 & & & &  \\
 \downarrow &  & \downarrow &  & \downarrow  \\
 & \Gamma_{\lambda}|_{\widetilde M\times \widetilde M} & & \pi &  \\
 \widetilde M\times \widetilde M&\longrightarrow &  T\mathbb R^q &
\longrightarrow &  \mathbb
 R^q \\
\end{array}
$$
\end{defn}

The
${\lambda}$-chord equivalence of $E_{\lambda}$ is a special
case of equivalence of projections studied by V. Goryunov
(\cite{Go1}, \cite{Go2}), as outlined below.

\begin{defn}
A {\bf projection} of a (smooth) submanifold $S$ from a
total space $E$ to the base $B$ of the bundle $p:E\rightarrow B$
is a triple
$$
\begin{array}{ccccc}
& \iota & & p & \\
S&\hookrightarrow &E& \rightarrow &B \\
\end{array}
$$
where $\iota$ is an embedding. A projection is called a {\bf
projection ``onto''} if the dimension of $S$ is not less than the
dimension of the base $B$.
\end{defn}

\begin{defn}
Two projections $S_i\hookrightarrow E_i \rightarrow B_i$ for
$i=1,2$ are {\bf equivalent} if the following diagram commutes
$$
\begin{array}{ccccc}
& \iota_1 & & p_1 & \\
S_1&\hookrightarrow &E_1& \rightarrow &B_1 \\
\downarrow& \iota_2 &\downarrow & p_2 & \downarrow\\
S_2&\hookrightarrow &E_2& \rightarrow &B_2 \\
\end{array}
$$
where vertical arrows indicate diffeomorphisms.
\end{defn}

A projection of $S$ onto $B$ defines a family of subvarieties in
the fibers of the bundle $p:E\rightarrow B$ parameterized by $B$:
$S_b=S\cap p^{-1}(b)$ for any $b\in B$. A germ of projection
$(S,q_0)\hookrightarrow (E,e_0) \rightarrow (B,b_0)$ can be
considered in a natural way as a deformation of the subvariety
$S_{b_0}$.

The germ of a bundle $E\rightarrow B$ can be identified with the
germ of the trivial bundle $\mathbb R^s\times \mathbb
R^p\rightarrow \mathbb R^p$. A germ of an embedded  smooth
submanifold $S$ can be described by the germ of the variety of
zeros of some mapping-germ $F:(\mathbb R^s\times \mathbb
R^p,(y_0,z_0))\rightarrow \mathbb R^t$. Then $S_{z_0}$ can be
identified with the germ of the variety of zeros of $F|_{\mathbb
R^s\times \{z_0\}}$.

If deformations $F, \widetilde{F}:(\mathbb R^s\times \mathbb
R^p,(y_0,z_0))\rightarrow (\mathbb R^t,0)$ of map-germs $f,
\widetilde{f}:(\mathbb R^s,y_0)\rightarrow (\mathbb R^t,0)$
(respectively)  are fiber $\mathcal K$-equivalent then the
following diagram commutes ($\Phi$, $Z$ indicate
diffeomorphism-germs and $pr$ indicate the projection):
$$
\begin{array}{ccccc}
&  & & pr & \\
F^{-1}(0)&\hookrightarrow & \mathbb R^s\times\mathbb R^p &
\longrightarrow &
\mathbb R^p \\
& & & &  \\
\downarrow &  & \downarrow \Phi&  & \downarrow  Z \\
& & & pr &  \\
\widetilde{F}^{-1}(0)&\hookrightarrow & \mathbb R^s\times\mathbb
R^p & \longrightarrow &
\mathbb R^p\\
\end{array}
$$
If the ideal of function-germs vanishing on $F^{-1}(0)$ is generated
by the components of $F$, then by Corollary \ref{V-def-ideal} the
inverse result is also true.

We remind that the group $\mathcal{A} =\mbox{Diff}(\mathbb R^m,0)\times \mbox{Diff}(\mathbb R^p,0)$ acts on map-germs $(\mathbb R^m,0)\to(\mathbb R^p,0)$ by composition on source and target,  with corresponding definitions for $\mathcal A$-equivalent and $\mathcal A$-simple (refer  to Definitions \ref{Keq} and \ref{Ksimp} for the group $\mathcal K$). Then,
from the above we have the following results:

\begin{prop}[\cite{Go1,Go2}]$F$ and $\widetilde F$ are fiber $\mathcal K$-equivalent if and only if
the projections of $F^{-1}(0)$ and
$\widetilde{F}^{-1}(0)$ onto $\mathbb R^p$ are $\mathcal A$-equivalent.\end{prop}
\begin{thm}[\cite{Go1}]\label{simple-simple}
If the germ of a projection $(F^{-1}(0),(0,0))\hookrightarrow
(\mathbb R^s\times \mathbb R^p,(0,0) )\rightarrow (\mathbb R^p,0)$
is $\mathcal A$-simple then $f=F|_{\mathbb R^s\times \{0\}}$ is $\mathcal
K$-simple.
\end{thm}
\begin{thm}[\cite{Mart, Mather}]\label{main} The map-germ $F:\mathbb R^s\times\mathbb R^p\to\mathbb R^t$ is a $\mathcal K$-versal deformation
of a rank-0 map-germ $f:\mathbb R^s\to\mathbb R^t$ of finite $\mathcal K$-codimension if and only if the projection-germ
of $F^{-1}(0)$  onto $\mathbb R^p$ is $\mathcal A$-stable (infinitesimally stable).
\end{thm}

By Theorems \ref{simple-simple} and \ref{main}, in order to classify stable singularities of
projections one considers deformations of three classes of
singularities: simple singularities of hypersurfaces (Table
\ref{ADE}), simple singularities of curves in a $3$-dimensional
space (Table \ref{Giusti}), simple singularities of a multiple
point on a plane (Table \ref{I}). We are interested in projections
"onto" when the projected submanifold $S=F^{-1}(0)$ is smooth and
the dimension of the base $B$ of the bundle is greater than 1.

In order to see in a more clear way how these three tables are applied to the
classification of singularities of affine equidistants, we now turn to the contact viewpoint.

\subsection{Singularities of contact}

Let $N_1, N_2$ be germs at $x$ of smooth $n$-dimensional submanifolds of
the space $\mathbb R^{q}$, with $2n\ge q$. We describe $N_1,N_2$  in the following way:
\begin{itemize}
\item
$N_1=f^{-1}(0)$, where $f:(\mathbb R^q,x)\rightarrow (\mathbb R^{q-n},0)$ is a submersion-germ,
\item
$N_2=g(\mathbb R^n)$, where $g:(\mathbb R^n,0)\rightarrow (\mathbb R^q,x)$ is an embedding-germ.
\end{itemize}

Let $\tilde{N_1}, \tilde{N_2}$ be another pair of germs at $\tilde{x}$ of smooth $n$-dimensional submanifolds of
the space $\mathbb R^{q}$, described in the same way.

\begin{defn}
The contact of $N_1$ and $N_2$ at $x$ is of the same {\bf contact-type} as the contact of $\tilde{N_1}$ and $\tilde{N_2}$ at $\tilde{x}$ if there exists a diffeomorphism-germ
$\Phi:(\mathbb R^q,x)\rightarrow (\mathbb R^q,\tilde{x})$ such that $\Phi(N_1)=\tilde{N_1}$ and $\Phi(N_2)=\tilde{N_2}$. We denote the contact-type of $N_1$ and $N_2$ at $x$ by $\mathcal K(N_1,N_2,x)$.
\end{defn}

\begin{defn}
A {\bf contact map} between submanifold-germs $N_1, N_2$ is the following map-germ
$f\circ g:(\mathbb R^n,0)\rightarrow (\mathbb R^{q-n},0)$.
\end{defn}

\begin{thm}[\cite{Mont}]
$\mathcal K(N_1,N_2,x) =\mathcal K(\tilde{N_1},\tilde{N_2},\tilde{x})$ if and only if the contact maps $f\circ g$ and $\tilde{f}\circ \tilde{g}$ are $\mathcal K$-equivalent.
\end{thm}

\begin{rem}If $N_1$ and $N_2$ are transversal at $x$ then it is obvious that the contact map $f\circ g:(\mathbb R^n,0)\rightarrow
(\mathbb R^{q-n},0)$ is a submersion-germ or a diffeomorphism-germ (when $q=2n$). \end{rem}

The interesting cases are when $N_1$ and $N_2$ are not transversal at $x_0$
$$T_{x_0}N_1 + T_{x_0}N_2\ne T_{x_0}\mathbb R^{q}.$$

\begin{defn}
We say that $N_1$ and $N_2$ are $k$-{\bf tangent} at $x_0$ if
$$\dim (T_{x_0}N_1 \cap T_{x_0}N_2)=k \ .$$
If $k$ is maximal, that is
$$ k=n=\dim(T_{x_0}N_1)=\dim(T_{x_0}N_2) \ ,$$
we say that $N_1$ and $N_2$ are {\bf tangent} at $x_0$.
\end{defn}

\begin{rem}\label{tan-par} In order to bring this definition into the context of affine equidistants, $E_{\lambda}(M)$, note that  $N_1=M^+$ and $N_2=\mathcal R_0^{\lambda}(M^-)$ are $k$-{\bf tangent} at $0$  if and only if $T_aM^+$ and $T_bM^-$ are $k$-{\bf parallel}, where $\lambda a +(1-\lambda) b = 0\in E_{\lambda}(M)$. \end{rem}

If $N_1$ and $N_2$ are $k$-tangent then we can describe germs of $N_1$ and $N_2$ at $0$ in the following way:
\begin{equation}\label{M}
 N_1=\{(y,z,u,v)\in \mathbb R^q: u=\phi(y,z), \ v=\psi(y,z)\},
\end{equation}
\begin{equation}\label{N}
N_2=\{(y,z,u,v)\in \mathbb R^q: z=\eta(y,v), \ u=\zeta(y,v)\},
\end{equation}
where $y=(y_1,\cdots,y_k)$, $z=(z_1,\cdots,z_{n-k})$, $u=(u_1,\cdots,u_{q+k-2n})$, $v=(v_1,\cdots,v_{n-k})$ and $(y,z,u,v)$ is a coordinate system on the affine space $\mathbb R^{q}$, $\phi=(\phi_1,\cdots,\phi_{q+k-2n})$, $\psi=(\psi_1,\cdots,\psi_{n-k})$, $\eta=(\eta_1,\cdots,\eta_{n-k})$, $\zeta=(\zeta_1,\cdots,\zeta_{q+k-2n})$ and $\phi_i, \psi_j, \eta_j, \zeta_i \in \mathcal M_q^2$ for $i=1,\cdots, q+k-2n$ and $j=1,\cdots,n-k$.

Then, the  contact map $\kappa_{N_1,N_2}:(\mathbb R^n,0)\rightarrow (\mathbb R^{q-n},0)$ is given by:
\begin{equation}\label{contactmap} \kappa_{N_1,N_2}(y,z)=(z-\eta(y,\psi(y,z)),\phi(y,z)-\zeta(y,\psi(y,z)))
\end{equation}

From the form of $\kappa_{N_1,N_2}$ we easily obtain the following fact
\begin{prop}
If $N_1$ and $N_2$ are $k$-tangent at $0$ then the corank of the contact map $\kappa_{N_1,N_2}$ is $k$.
\end{prop}

We can interpret the contact between two $k$-tangent $n$-dimensional submanifolds $N_1, N_2$ of $\mathbb R^{q}$ as the contact between tangent $k$-dimensional submanifolds $P_{N_1}$ and $P_{N_2}$ of $N_1$ and $N_2$, respectively, in a smooth $q-2n+2k$-dimensional submanifold $S$ of $\mathbb R^{q}$.
These submanifolds are constructed in the following way:

Let $H$ be a smooth $q+k-n$-dimensional submanifold-germ on $\mathbb R^{q}$ which contains $N_1$ and is transversal to $N_2$ at $0$. Then $P_{N_2}=H\cap N_2$ is a smooth $k$-dimensional submanifold on $N_2$.

Let $G$ be a smooth $q+k-n$-dimensional submanifold-germ on $\mathbb R^{q}$ which contains $N_2$ and is transversal to $N_1$ at $0$. Then $P_{N_1}=G\cap N_1$ is a smooth $k$-dimensional submanifold on $N_1$.

$P_{N_1}$ and $P_{N_2}$ are tangent at $0$ and they are contained in the smooth $q-2n+2k$-dimensional submanifold-germ $S=H\cap G$.\\

The contact between $N_1$ and $N_2$ at $0$ can now be described as the contact between $P_{N_1}$ and $P_{N_2}$ at $0$, which defines a rank-$0$ map
\begin{equation}\label{rank0map}
\kappa_{P_{N_1},P_{N_2}}:(\mathbb R^k,0)\rightarrow (\mathbb R^{k-(2n-q)},0) \ .
\end{equation}

Although in general $P_{N_1}$ and $P_{N_2}$ depend on the choices of $H$ and $G$, the contact type of $P_{N_1}$ and $P_{N_2}$ does not depend on these choices.
This means that if $\tilde N_1, \tilde N_2$ is another pair of germs at $0$ of smooth $n$-dimensional submanifold of $\mathbb R^{q}$ then we have the following result.
\begin{prop}\label{contactreduction}
$\mathcal K(N_1,N_2,0)=\mathcal K(\tilde N_1,\tilde N_2,0)$ if and only if $$\mathcal K(P_{N_1},P_{N_2},0)=\mathcal K(P_{\tilde N_1},P_{\tilde N_2},0).$$
\end{prop}

\begin{proof}
It is easy to see that in general $H$ can be described in the following way:
\begin{equation}\label{H}
v=\psi(y,z)+A(y,z,u,v)(u-\phi(y,z)),
\end{equation}
and $G$ can be described in the following way:
\begin{equation}\label{G}
z=\eta(y,v)+B(y,z,u,v)(u-\zeta(y,v)),
\end{equation}
where $A=(a_{ij})_{i=1,\cdots,q+k-2n}^{j=1,\cdots,n-k}, B=(b_{ij})_{i=1,\cdots,q+k-2n}^{j=1,\cdots,n-k}$ and $a_{ij}, b_{ij}$ are smooth function-germs on $\mathbb R^{q}$.

Thus $S=H\cap G$ is given by (\ref{H}) and (\ref{G}).

 $P_{N_1}$ is given by (\ref{H}), (\ref{G}), and $u=\phi(y,z)$ and $P_{N_2}$ is given by (\ref{H}), (\ref{G}) and $u=\zeta(y,v)$.

On the other hand we can also describe $N_1$ by (\ref{H}) and $u=\phi(y,z)$ and $N_2$ by (\ref{G}) and $u=\zeta(y,v)$.
Then it is easy to see that contact maps are the same after a suitable suspension.
\end{proof}

In view of Proposition \ref{contactreduction}, it is enough to classify the rank-$0$ map-germs of the form (\ref{rank0map}) with respect to the group $\mathcal K$.

\section{Stable singularities of affine equidistants}

Since our goal is to classify singularities of affine equidistants of $n$-dimensional submanifold $M$ of $\mathbb R^q$, we substitute submanifold-germs $N_1$ and $N_2$ of the previous section  by $N_1=M^+$ and $N_2=\mathcal R_0^{\lambda}(M^-)$, or equivalently by $N_1=M^-$ and $N_2=\mathcal R_0^{1-\lambda}(M^+)$, where $M^+$ and $M^-$ are germs of $M\subset\mathbb R^q$ at  points $a^+ \neq a^-\in M\subset\mathbb R^q$, such that $\lambda a^+ + (1-\lambda)a^-=0$.

First, we state the following definition and theorem:

\begin{defn} A mapping $\psi : N^m \to \mathbb R^q$  is \emph{locally stable}  at $p \in N^m$ if there exists a neighbourhood $W_p$ of $\psi$ in the space  $C^{\infty}(N^m, \mathbb R^q)$  of $C^{\infty}$-mappings from $N^m$ into
$\mathbb R^q$  with the  Whitney  $C^{\infty}$-topology, and neighbourhoods $U_p$ around $p$ and $V_p$ around $\psi(p)$ such that for all $\phi \in W_p,$  it follows that $\phi: U_p \to V_p$ is $\mathcal A$- equivalent to $\psi: U_p \to V_p,$  where $\mathcal A= Diff(U_p)\times Diff(V_p)$ (see \cite{GG}).
\end{defn}

\begin{thm}\label{genAstable}
For a residual set of embeddings $\iota : M^n \rightarrow \mathbb R^q$ the map $ {\pi}_{\lambda}\circ (\iota\times \iota):M\times M\setminus \Delta \rightarrow \mathbb R^q$ is locally stable whenever the pair $(2n,q)$ is a pair of nice dimensions, where $\Delta$ is the diagonal in $M\times M$.
\end{thm}

\begin{proof}
From the diagram of maps
$$
\begin{array}{ccccccccc}\label{diagram}
 & \iota\times\iota&  & \pi_{\lambda} & \\
M\times  M & \longrightarrow & \mathbb R^q\times
\mathbb R^{q} & \to &
\mathbb R^{q} \ ,\\
\end{array}
$$
we obtain the diagram of $r$-jet maps
$$
\begin{array}{ccccccccc}\label{diagram}
 & j^r(\iota\times\iota)&  & (\pi_{\lambda})_* & \\
M\times  M & \longrightarrow & J^r(M\times M,\mathbb R^q\times
\mathbb R^{q}) & \to &
J^r(M\times M,\mathbb R^{q}) \ .\\
\end{array}
$$
A typical fiber of $J^r(M\times M,\mathbb R^{q})$ is $J^r_0(M\times M,\mathbb R^{q})$, the space of (degree $\leq r$)-polynomial map-germs $\mathbb R^n\times\mathbb R^n\to \mathbb R^q$, vanishing at $0$.

Let $\{W_1, \ldots, W_s\}$  be the finite set of all $\mathcal K$ simple orbits in $J^r(M\times M, \mathbb R^q),$ and let
$\{W_{s+1},\ldots, W_{t}\}$ be a finite stratification of the complement of the union of simple orbits $W_1\cup\ldots\cup W_s.$
This stratification exists because these are semialgebraic sets. We denote by $\mathcal S=\{{W_j}\}_{1\leq j\leq t}$ the resulting stratification of $J^r(M\times M, \mathbb R^q).$ Because $ (\pi_{\lambda})_*$ is a submersion,  $(\pi_{\lambda})_*^{-1} W_j=W_j^{*}$ is a submanifold  of $J^r(M\times M,\mathbb R^q\times \mathbb R^{q})$, for all $j=1,\ldots,t$, so that
${\mathcal S}^{*}=\{W_j^{*}\}_{1\leq j\leq t}$ is a stratification of this space.

Furthermore,
\begin{equation}\label{iff}  j^r(\iota\times\iota) \ \pitchfork \ {\mathcal S^*} \ \iff \ j^r(\pi_{\lambda}\circ(\iota\times\iota))  \ \pitchfork \ {\mathcal S} \ , \end{equation} (where transversality to $\mathcal S$ (respectively to $\mathcal S^{*}$) means transversality of $j^r(\iota\times \iota)$  (respectively  $j^r(\pi_{\lambda}\circ(\iota\times \iota))$ to each stratum of the corresponding stratification.

On the other hand, under the natural identification
$$j^r(\iota\times\iota)|_{M\times M\setminus\Delta} \ \simeq \  { }_{2}j^r\iota \ \subset \
 { }_{2}J^r(M,\mathbb R^q) \ ,$$
where $ { }_{2}J^r(M,\mathbb R^q)$ is the space of double $r$-jets, we can apply the Multijet Transversality Theorem \cite{GG} to get that, for each $W_{j}^{*}$ in  ${ }_{2}J^r(M,\mathbb R^q),$ the set of immersions
$$\mathcal {R}_{W_j}=\{\iota: M \to \mathbb R^q\,|  {}_{2}j^r \iota \pitchfork W_j^{*}\}$$ is residual. Then, the set
$$\mathcal R = \cap_{j=1}^{t} \mathcal  R_{W_j}$$ is also residual.

Now, it follows from equation (\ref{iff}) that  $j^r(\pi_{\lambda}\circ (\iota \times \iota)) \pitchfork W_j$, for all $\iota \in \mathcal R$, for all $j=1,\ldots, t.$ When $(2n,q)$ is a pair of nice dimensions, this implies that $j^r(\pi_{\lambda}\circ (\iota \times \iota))$ is transversal to all $\mathcal K$ orbits in $J^r(M\times M, \mathbb R^q),$ which says that this mapping is locally stable (see \cite{GG, Mather}).
\end{proof}

\begin{thm}[\cite{Mather}]\label{nicedim} The nice dimensions for pairs $(2n,q)$  are: \\
(i) $n<q=2n, \ n\leq 4$ \\
(ii) $n<q=2n-1, \ n\leq 4$ \\
(iii) $n<q=2n-2, \ n\leq 3$ \\
(iv) $n<q\leq 2n-3, q\leq 6$
\end{thm}

Thinking locally, denote two distinct germs of embedding $\iota:M^n\to \mathbb R^q$ by $\iota^+: (\mathbb R^n,0)\to (\mathbb R^q,a^+)$ and $\iota^-: (\mathbb R^n,0)\to (\mathbb R^q,a^-)$, and by
\begin{equation}\label{tildepi}
\tilde\pi_{\lambda} = \pi_{\lambda}\circ(\iota^+\times\iota^-) :  (\mathbb R^{2n},0)\to (\mathbb R^q,0) \  ,
\end{equation}
the restriction of $\pi_{\lambda}$  to $M^+\times M^-$. Then,
 recalling the notation of (\ref{M})-(\ref{N}), $\tilde\pi_{\lambda}$ is given by
\begin{equation}\label{localtildepi}
\tilde\pi_{\lambda}: (y,z,\tilde{y},v)\mapsto (\tilde\pi_{\lambda}^1(y,\tilde{y}), \tilde\pi_{\lambda}^2(z,\tilde{y},v), \tilde\pi_{\lambda}^3(y,z,\tilde{y},v), \tilde\pi_{\lambda}^4(y,z,v))
\end{equation}
where $y,\tilde{y}\in\mathbb R^k$, $z,v\in\mathbb R^{n-k}$, and
\begin{equation}\label{pitilde1}
\tilde\pi_{\lambda}^1(y,\tilde{y})= \lambda y +(1-\lambda)\tilde{y},
\end{equation}
\begin{equation}\label{pitilde2}
\tilde\pi_{\lambda}^2(z,\tilde{y},v)=
\lambda z +(1-\lambda)\eta(\tilde{y},v),
\end{equation}
\begin{equation}\label{pitilde3}
\tilde\pi_{\lambda}^3(y,z,\tilde{y},v)=
\lambda\phi(y,z)+(1-\lambda)\zeta(\tilde{y},v),
\end{equation}
\begin{equation}\label{pitilde4}
\tilde\pi_{\lambda}^4(y,z,v)=\lambda\psi(y,z)+(1-\lambda)v.
\end{equation}

 Let
$$\kappa_{\lambda}: (\mathbb R^n,0)\to (\mathbb R^{q-n},0)$$ denote the the contact-map between $M^+$ and $\mathcal R_0^{\lambda}(M^-)$.
We have:
\begin{prop} \label{localrings}
Local rings $\displaystyle{\frac{\mathcal E_{2n}}{\tilde\pi_{\lambda}^{\ast}(\mathfrak m_q)}}$ and
$\displaystyle{\frac{\mathcal E_{n}}{\kappa_{\lambda}^{\ast}(\mathfrak m_{q-n})}}$ are isomorphic.
\end{prop}
\begin{proof} From (\ref{localtildepi}), we have that
\begin{equation}
\displaystyle{\frac{\mathcal E_{2n}}{\tilde\pi_{\lambda}^{\ast}(\mathfrak m_q)}}\simeq \frac{{\mathcal E}_{(y,z,\tilde{y}, v)}}{\langle \tilde\pi_{\lambda}^1(y,\tilde{y}), \tilde\pi_{\lambda}^2(z,\tilde{y},v), \tilde\pi_{\lambda}^3(y,z,\tilde{y},v), \tilde\pi_{\lambda}^4(y,z,v)\rangle}\nonumber
\end{equation}
so that, using (\ref{pitilde1})-(\ref{pitilde4}), this is isomorphic to
\begin{equation}
\frac{{\mathcal E}_{(y, z)}}{\langle z +\frac{(1-\lambda)}{\lambda}\eta(-\frac{\lambda}{(1-\lambda)}y,-\frac{\lambda}{(1-\lambda)}\psi(y,z)), \phi(y,z)+\frac{(1-\lambda)}{\lambda}\zeta(-\frac{\lambda}{(1-\lambda)}y,-\frac{\lambda}{(1-\lambda)}\psi(y,z)) \rangle}\nonumber
\end{equation}
and, using (\ref{contactmap}) for $N_1=M^+$ and $N_2= \mathcal R_0^{\lambda}(M^-)$, we see that the above local ring is isomorphic to $\displaystyle{\frac{\mathcal E_{n}}{\kappa_{\lambda}^{\ast}(\mathfrak m_{q-n})}}$.
\end{proof}

On the other hand, we remind from Remark \ref{tan-par} that $k$ is the degree of tangency of  $M^+$ and $\mathcal R_0^{\lambda}(M^-)$ and therefore $k$ is the degree of parallelism of  $T_{a^+}M^+$ and $T_{a^-}M^-$, where $\lambda a^+ +(1-\lambda) a^- = 0\in E_{\lambda}(M)$, so that,
denoting by
$$\theta_{\lambda}: (\mathbb R^k,0)\to (\mathbb R^{k-(2n-q)},0)$$
the reduced (rank-0) contact map $\theta_{\lambda}=\kappa_{P_{N_1},P_{N_2}}$, for $N_1=M^+$ and $N_2=\mathcal R_0^{\lambda}(M^-)$, from Proposition \ref{contactreduction} we have the following
\begin{cor}\label{locrings2} The local rings  $\displaystyle{\frac{\mathcal E_{n}}{\kappa_{\lambda}^{\ast}(\mathfrak m_{q-n})}}$ and  $\displaystyle{\frac{\mathcal E_{k}}{\theta_{\lambda}^{\ast}(\mathfrak m_{k-(2n-q)})}}$ are isomorphic.
\end{cor}

Thus, by Theorems  \ref{main} and \ref{genAstable},  Proposition \ref{localrings} and Corollary \ref{locrings2}, for the local classification of  stable singularities of affine equidistants,
we need to determine every rank-$0$ $\mathcal K$-simple map-germ
\begin{equation}\label{1}\theta_{\lambda} : (\mathbb R^k,0)\to (\mathbb R^l,0) \  ,\end{equation}
that admits a $\mathcal K$-versal deformation $F_{\lambda}: \mathbb R^k\times\mathbb R^q\to \mathbb R^l$, so that
\begin{equation}\label{2}
\tilde\pi_{\lambda} :  (F_{\lambda})^{-1}(0)=(\mathbb R^{2n},0)\to (\mathbb R^q,0) \
\end{equation}
is an $\mathcal A$-stable map. Here, $\theta_{\lambda}=\kappa_{P_{N_1},P_{N_2}}$, for $N_1=M^+$ and $N_2=\mathcal R_0^{\lambda}(M^-)$, and  $\tilde\pi_{\lambda}$ is the restriction of $\pi_{\lambda}$  to $M^+\times M^-$, so that
\begin{equation}\label{numbers}
l=k-(2n-q) \ \ , \ \ 1\leq k \leq n \ \ , \ \  2n \geq q > n\ \ ,
\end{equation}
for any pair $(2n,q)$ in the nice dimensions (Theorem \ref{nicedim}).

In other words, we unfold the map-germ $\theta_{\lambda}$ with $m$ parameters,
\begin{equation}\label{3}\tilde\pi_{\lambda} :  (\mathbb R^{m}\times\mathbb R^k,0)\to (\mathbb R^m\times\mathbb R^l,0) \ \ , \ \ (w,y)\mapsto (w,u(w,y)) \ ,\end{equation}
where $m = 2n-k$, so that $\tilde\pi_{\lambda}$ is $\mathcal A$-stable. Thus, in each case, we look for the rank-0 $\mathcal K$-simple map-germs $\theta_{\lambda}$ that can be unfolded with $m=2n-k$ parameters so that its $\mathcal K_e$-codimension $\mu$ is such that
\begin{equation}\label{codimmax} \mu\leq l+m=q \ . \end{equation}

The list of $\mathcal K$-simple map-germs $\theta_{\lambda}$ is presented in Tables $1$, $2$ and $3$, in section $2$ above. Thus, for classifying the stable singularities of affine equidistants of smooth submanifolds $M^n\subset\mathbb R^q$, all we have to do is read those Tables with respect to the numbers $k$, $l$ and $\mu$, subject to conditions  (\ref{numbers}) and (\ref{codimmax}) for each pair $(2n,q)$  in the nice dimensions.

In this way, we arrive at our main result, as follows.

\subsection{All possible stable singularities in the nice dimensions}\label{result}

First, remind the definition of $k$-parallelism, cf. (\ref{k//}). Then, we have:

\begin{thm}\label{mainresult} Let $M^n\subset\mathbb R^q$ be a smooth closed submanifold of the affine space, such that  $2n\geq q$ and  $(2n,q)$ is a pair of nice dimensions, as listed in Theorem \ref{nicedim}. Then, the possible stable singularities of the $\lambda$-affine equidistant $E_{\lambda}(M)\subset\mathbb R^q$ are listed case by case, as below. \end{thm}

\

\noindent{\it Curves}:

\

\noindent In this case, we have curves in $\mathbb R^2$ and the rank-$0$ contact map is $\theta_{\lambda}:\mathbb R\to\mathbb R$,  $\mu\leq 2$. From Table $1$, the stable singularities of affine equidistants can be of type $A_1$ and $A_2$.

\

\noindent{\it Surfaces}:

\

\noindent ($1$) $M^2\subset\mathbb R^3$.

 $2$-parallelism.  $\theta_{\lambda}:\mathbb R^2\to\mathbb R$,  $\mu\leq 3$.

 $E_{\lambda}(M)$ with stable singularities of types $A_1$, $A_2$ and $A_3$.

\

\noindent  ($2$) $M^2\subset\mathbb R^4$.

\noindent (i) $1$-parallelism.  $\theta_{\lambda}:\mathbb R\to\mathbb R$,  $\mu\leq 4$.

$E_{\lambda}(M)$ with stable singularities of types $A_1$, $A_2$, $A_3$ and $A_4$.

\noindent (ii) $2$-parallelism.  $\theta_{\lambda}:\mathbb R^2\to\mathbb R^2$,  $\mu\leq 4$.

 $E_{\lambda}(M)$ with stable singularities of types $C^{\pm}_{2,2}$.

\

\noindent{\it $3$-manifolds}:

\

\noindent ($1$) $M^3\subset\mathbb R^4$.

 $3$-parallelism. $\theta_{\lambda}:\mathbb R^3\to\mathbb R$,  $\mu\leq 4$.

$E_{\lambda}(M)$ with stable singularities of types $A_1$, ..., $A_4$ and $D_4^{\pm}$.

\

\noindent ($2$) $M^3\subset\mathbb R^5$.

\noindent (i) $2$-parallelism. $\theta_{\lambda}:\mathbb R^2\to\mathbb R$,  $\mu\leq 5$.

$E_{\lambda}(M)$ with stable singularities of types $A_1$, ..., $A_5$, $D_4^{\pm}$, $D_5^{\pm}$.

\noindent (ii) $3$-parallelism. $\theta_{\lambda}:\mathbb R^3\to\mathbb R^2$,  $\mu\leq 5$.

$E_{\lambda}(M)$ with stable singularites of types $S_5$.

\

\noindent ($3$) $M^3\subset\mathbb R^6$.

\noindent (i) $1$-parallelism.  $\theta_{\lambda}:\mathbb R\to\mathbb R$,  $\mu\leq 6$.

$E_{\lambda}(M)$ with stable singularities of types $A_1$, ..., $A_6$.

\noindent (ii) $2$-parallelism. $\theta_{\lambda}:\mathbb R^2\to\mathbb R^2$,  $\mu\leq 6$.

$E_{\lambda}(M)$ with stable singularities of types $C^{\pm}_{2,2}$, $C^{\pm}_{2,3}$, $C^{\pm}_{2,4}$, $C^{\pm}_{3,3}$, $C_{6}$.

\noindent (iii) $3$-parallelism. No stable singularities for $E_{\lambda}(M)$.

\

\noindent{\it $4$-manifolds}:

\

\noindent ($1$) $M^4\subset\mathbb R^5$.

 $4$-parallelism. $\theta_{\lambda}:\mathbb R^4\to\mathbb R$,  $\mu\leq 5$.

$E_{\lambda}(M)$ with stable singularities of types $A_1$, ..., $A_5$, $D_4^{\pm}$, $D_5^{\pm}$.

\

\noindent ($2$) $M^4\subset\mathbb R^6$: The map $\tilde\pi_{\lambda}:\mathbb R^8\to\mathbb R^6$ is not in nice dimensions.

\

\noindent ($3$) $M^4\subset\mathbb R^7$.

\noindent (i) $2$-parallelism.  $\theta_{\lambda}:\mathbb R^2\to\mathbb R$,  $\mu\leq 7$.

$E_{\lambda}(M)$ with stable singularities $A_1$, ..., $A_7$, $D_4^{\pm}$, ..., $D_7^{\pm}$, $E_6$, $E_7$.

\noindent (ii) $3$-parallelism. $\theta_{\lambda}:\mathbb R^3\to\mathbb R^2$,  $\mu\leq 7$.

$E_{\lambda}(M)$ with stable singularities of types $S_5$, $S_6$, $S_7$, $T_7$, $\tilde T_7$.

\noindent (iii) $4$-parallelism. No stable singularities for $E_{\lambda}(M)$.

\

\noindent ($4$) $M^4\subset\mathbb R^8$.

\noindent (i) $1$-parallelism.  $\theta_{\lambda}:\mathbb R\to\mathbb R$,  $\mu\leq 8$.

$E_{\lambda}(M)$ with stable singularities of types $A_1$, ..., $A_8$.

\noindent (ii) $2$-parallelism. $\theta_{\lambda}:\mathbb R^2\to\mathbb R^2$,  $\mu\leq 8$.

$E_{\lambda}(M)$ with stable singularities of types

$C^{\pm}_{2,2}$, $C^{\pm}_{2,3}$, $C^{\pm}_{2,4}$,
$C^{\pm}_{2,5}$, $C^{\pm}_{2,6}$, $C^{\pm}_{3,3}$, $C^{\pm}_{3,4}$, $C^{\pm}_{3,5}$, $C^{\pm}_{4,4}$, $C_{6}$, $C_{8}$, $F_{7}$, $F_{8}$.

\noindent (iii) $3$-parallelism, $4$-parallelism. No stable singularities for $E_{\lambda}(M)$.

\

\noindent{\it $5$-manifolds}:

\

\noindent ($1$) $M^5\subset\mathbb R^6$.

 $5$-parallelism. $\theta_{\lambda}:\mathbb R^5\to\mathbb R$,  $\mu\leq 6$.

$E_{\lambda}(M)$ with stable singularities $A_1$, ..., $A_6$, $D_4^{\pm}$, $D_5^{\pm}$,  $D_6^{\pm}$,  $E_6$.

\

\noindent ($2$) For all other embeddings $M^5\subset\mathbb R^q$,  no map $\tilde\pi_{\lambda}$ in nice dimensions.

\

\noindent{\it $n$-manifolds, $n\geq 6$}: No map $\tilde\pi_{\lambda}$ in nice dimensions.

\

\end{document}